\documentclass[final]{elsarticle}

\usepackage{booktabs}
\usepackage{subcaption}
\usepackage{xcolor}
\usepackage{graphicx,bm}

\usepackage{amsmath, amsthm, amsfonts}
\usepackage{stmaryrd}

\definecolor{OliveGreen}{rgb}{0.3, 0.4, .1}

\def\a{\alpha}
\def\b{\beta}

\def\f{\varphi}

\def\g{\gamma}

\def\s{\sigma}
\def\t{\tau}

\def\Ccal{\mathcal{C}}

\def\Ical{\mathcal{I}}

\def\Pcal{\mathcal{P}}

\def\Rbb{\ensuremath{\mathbb{R}}}

\def\mbfB{\ensuremath{\mathbf{B}}}
\def\mbfc{\ensuremath{\mathbf{c}}}
\def\mbfd{\ensuremath{\mathbf{d}}}

\def\mbfv{\ensuremath{\mathbf{v}}}

\def\mbfw{\ensuremath{\mathbf{w}}}

\def\mbfz{\ensuremath{\mathbf{z}}}

\def\mbf0{\ensuremath{\bm{0}}}

\newtheorem{thm}{Theorem}
\newtheorem{lem}{Lemma}
\theoremstyle{definition}

\theoremstyle{remark}
\newtheorem{rem}{Remark}

\def\<{\langle}
\def\>{\rangle}

\def \jump[#1]{\llbracket #1 \rrbracket}
\def \avg[#1]{\{ \! \{ #1 \}\!\}}
\DeclareMathOperator\erf{erf}
\DeclareMathOperator\erfc{erfc}

\def\by{\times}

\def\->{\rightarrow}

\begin{document}
\title{{A Posteriori} Error Estimation for Multi-Stage Runge-Kutta IMEX Schemes}
\author[NMU]{Jehanzeb H.~Chaudhry}
\ead{jehanzeb@unm.edu}
\author[WTAMU]{J.B.~Collins}
\ead{jcollins@wtamu.edu}
\author[NMU,Sandia]{John N. Shadid}
\ead{jnshadi@sandia.gov}

\address[NMU]{Department of Mathematics and Statistics, The University of New Mexico, Albuquerque, NM 87131}
\address[WTAMU]{Department of Mathematics, West Texas A\&M University, Canyon, TX 79016}
\address[Sandia]{Computational Mathematics Department, Sandia National Laboratories, Albuquerque, NM 87123}

\begin{abstract}
	Implicit-Explicit (IMEX) schemes are widely used for time integration methods for approximating solutions to a large class of problems.  In this work, we develop accurate \emph{a posteriori} error estimates of a quantity-of-interest for approximations obtained from multi-stage IMEX schemes.  This is done by first defining a finite element method that is nodally equivalent to an IMEX scheme, then using typical methods for adjoint-based error estimation.  The use of a nodally equivalent finite element method allows a decomposition of the error into multiple components, each describing the effect of a different portion of the method on the total error in a quantity-of-interest.  	
\end{abstract}


\maketitle

\section{Introduction}

In this paper 
we consider \emph{a posteriori} error analysis for multi-stage implicit-explicit (IMEX) schemes applied to autonomous  nonlinear ODEs,
\begin{equation}
\begin{cases}
\dot{y}(t) = f(y(t)) + g(y(t)), \quad t \in (0,T],\\
y(0) = y_0.
\end{cases}
\label{eq:model_ode}
\end{equation}
Here $y(t) \in \mathbb{R}^m$, $\dot{y}  = d{y}/ d{t}$ and
$f,g:\mathbb{R}^m \rightarrow \mathbb{R}^m$.  The right hand side of \eqref{eq:model_ode} is split such that  $g$  represents a much smaller time scale than  $f$, and is often referred to as the ``stiff'' term.  Systems of the form
\eqref{eq:model_ode} often arise from spatial discretization of
partial differential equations,  for example, convection-diffusion-reaction equations, and hyperbolic systems with relaxation~\cite{PR2000}, where $f$  represents the convection term and $g$  represents the diffusion or relaxation term.

As opposed to \emph{a priori} error bounds, \emph{a posteriori} error estimates provide an accurate  computation of the discretization error in a particular approximation.  Accurate error estimation is a critical component of numerical simulations, being useful for reliability, uncertainty
quantification and adaptive error control.  In particular, we consider goal-oriented \emph{a posteriori} error estimation.  Often, the aim of a numerical simulation is to compute the value of a linear functional of the solution, a so called quantity-of-interest (QoI) defined as,
\begin{equation}
\label{eq:qoi}
Q(y) \equiv (y(T), \psi ),
\end{equation}
where $\psi \in \mathbb{R}^m$ and $(\cdot, \cdot)$ denotes the standard Euclidean inner product.  In this paper we employ adjoint based \emph{a posteriori} analysis to
quantify error for a given QoI.  Adjoint based error estimation is
widely used for a host of numerical methods including finite elements,
time integration, multi-scale simulations and inverse problems~\cite{estep_sinum_95,eehj_book_96,Estep:Larson:00,
  AO2000,BangerthRannacher,barth04,beckerrannacher,
  Giles:Suli:02, YCLP2004, CEG+2015, jdt13, VRAS2015}. The error estimate weights
computable residuals of the numerical solution with the
solution of an adjoint problem to quantify the accumulation and propagation of
error.  Moreover, the estimate also identifies different contributions
to the total error.

Developing accurate and stable time  integration
of systems of the form \eqref{eq:model_ode}  is challenging as the
term $g$ represents time scales which are often order of magnitudes smaller
than the time scales for the component $f$~\cite{IMEXRKHyperbolicStiffRelax2005}.  IMEX schemes offer an attractive option for such systems. 
These schemes treat the $f$ component explicitly while treating the $g$ component implicitly.  Hence, such schemes attempt to minimize the computational cost by balancing the number of nonlinear solves needed for an implicit scheme with the small time step required to maintain stability with an explicit scheme.

IMEX schemes are widely used for time integration methods for approximating solutions to a large class of problems~\cite{IMEXRKHyperbolicStiffRelax2005,HuRu07,StabIMEXHyperbolicStiffRX2011,
IMEXRadHydroNonLinElim2010,KadiogluKnollHydroNonlinearHeat2010,SvardMishraIMEXFlowStiffSource2011,IMEXBlackHoles2011,GenRelativisticRadHydroIMEX2012,ARSW95,ARS97,CarpenterKennedy2005,CS2010,CJS+2014,ZSZ2015,ZSB2014}.  IMEX schemes may be divided into two classes:
multi-step IMEX schemes and multi-stage Runge-Kutta IMEX
schemes (also termed here IMEX RK). Multi-step IMEX schemes  are a generalization of
multi-step schemes like Adams-Bashforth or  backward differentiation
formulas (BDF) and utilize solution values from previous
time nodes to form the solution at the current time node. 
IMEX multi-step methods  can be designed with
various stability properties such as A-stability, L-stability
and strong-stability preservation (SSP) that are 
desirable for a wide range of challenging systems \cite{HuRu07}.
Due to the inherent data dependencies  of higher-order  multi-step methods these techniques 
often require start-up procedures  that employ lower-order approximations with smaller scaled time-step sizes to 
ramp up to the full formal approximation order. Additionally, to restart a 
calculation, additional old time step values for the solution are required. 
In contrast multi-stage Runge-Kutta in general, and multi-stage  IMEX RK schemes 
specifically do not require ramping procedures and are self-restarting methods. This is a desirable feature for 
long-time-scale integrations with high-order accuracy. Additionally, multi-stage  IMEX RK methods can also be designed with
various stability properties such as A-stability, L-stability and strong-stability preservation (SSP) \cite{HuRu07}. For an extensive  comparison of 
multistep and multi-stage IMEX methods see \cite{HuRu07} and the references contained therein. In general both IMEX multi-step and multi-stage methods have
significant advantages and potential limitations when considering issues such as stability, accuracy, memory usage, restarting, etc. for specific classes of challenging multiphysics systems. 
However the ability to allow both explicit and implicit evaluation of operators in a multiple-time-scale multiphsyics system integration, in a well structured general mathematical method, is very appealing as described above. 
This paper focuses on what we believe is the first development of quantity-of-interest focused \emph{a posteriori} error estimation   for IMEX RK type methods and complements our earlier work on IMEX multistep methods \cite{CEG+2015}.

Adjoint-based \emph{a posteriori} error estimation is widely applied to finite element approximations, as a variational formulation is needed to compute the error, however, there is some recent work which considers finite difference methods.  Explicit multi-stage and multi-step time integrators are considered in \cite{jdt13} and multi-step IMEX schemes were considered in \cite{CEG+2015}.  Error estimates for the Lax-Wendroff scheme were developed in \cite{Collins:2014} and certain Godunov methods were analyzed in \cite{Barth:2007, barth:2002, larson:2000}.

  In this paper, we derive \emph{a posteriori} error estimates for the multi-stage Runge-Kutta IMEX  schemes.  To derive an estimate, we must first represent the scheme as a finite element method.  This is done by developing a finite element method that is ``nodally equivalent" to a certain IMEX scheme.  A nodally equivalent finite element approximation agrees with the IMEX approximation at the discretization nodes, while still being defined on the entire temporal domain.  The error representation formula is then derived using well defined methods.  In addition, we decompose the error estimate into components representing different contributions to the error.  In this way we are able to discuss the contribution of the implicit and explicit portions of the method to the total error in the quantity-of-interest.

 The paper is organized as follows. In \S \ref{sec:prelim} we discuss
the Runge-Kutta IMEX  schemes and formulate an equivalent finite
element method. This equivalence allows us to carry out \emph{a posteriori}
analysis in \S \ref{sec:apost_ana}. Numerical examples arising from the
discretization of partial differential equations are presented in \S
\ref{sec:num_examples}. The examples are associated with linear and nonlinear
convection diffusion type systems, a nonlinear Burger's equation with dissipation, 
and a coupled system from magnetohydrodynamics (MHD) 
representing propagation of an Alfven wave in a viscous conducting fluid.

\section{Preliminaries and Equivalent Finite Element Method}	
\label{sec:prelim}
	
	In this section, we introduce some notation, then give a brief review of generic $\nu$-stage IMEX Runge-Kutta schemes.  For a more complete discussion see \cite{PR2000,IMEXRKHyperbolicStiffRelax2005, ARS97}.  Generic continuous Galerkin finite element methods are introduced for solving \eqref{eq:model_ode}, and then the idea of nodal equivalency is explained and the equivalent finite element method is derived.  Convergence properties of finite element methods for time integration are discussed in \cite{eehj_book_96}. The section is concluded by showing second order convergence of the nodally equivalent finite element solution.
	
	\subsection{Notation and Generic IMEX Runge-Kutta Schemes} \label{ssec:general:IMEX}
		  We begin with some notation.  All approximations discussed are based on a discretization defined by the nodes
		\[0 = t_0 < t_1 < \ldots < t_n < \ldots < t_N = T,\]
		with time step $k_n = t_{n+1} - t_n$.  The IMEX approximation at each node is denoted by $Y_n \approx y(t_n)$.  Finally, for brevity we introduce notation for the subintervals $I_n := [ t_{n}, t_{n+1}]$ and for sub-discretization nodes $t_{n+\t} := t_n + k_n\t$ with $\t \leq 1$.
		
		Multi-stage IMEX schemes are defined by two Butcher tableaus, one for the explicit method and another for the implicit method, 
	\begin{align}
		\begin{tabular}{l|l}
		$\mbfc$ & $A$           \\ \hline
		      & $\mbfw$    
		\end{tabular}
	\hspace{.3in}
		\begin{tabular}{l|l}
		$\mbfd$ & $B$          \\ \hline
		      & $\tilde{\mbfw}$    
		\end{tabular}, \label{eq:general:Butcher}
	\end{align}
	where $A  \in \Rbb^{\nu \by \nu}$ and $\mbfc, \mbfw \in \Rbb^\nu$ define the explicit method and $B \in \Rbb^{\nu \by \nu}$ and $\mbfd, \tilde{\mbfw} \in \Rbb^\nu$ define the implicit method. The components of $A,\mbfc, \mbfw , B,\mbfd$ and $\tilde{\mbfw} $ are denoted as  $a_{ij}, c_i, w_i, b_{i,j}, d_i $ and $\tilde{w}_i$ respectively for $1 \leq i,j \leq \nu$. The update formula of the IMEX scheme is given by,
	\begin{align}
		\tilde{Y}_i &= Y_n  + k_n \sum_{j=1}^{i-1} a_{ij} f(\tilde{Y}_j) + k_n \sum_{j=1}^\nu b_{ij} g(\tilde{Y}_j), \label{eq:stage:vars} \\
		Y_{n+1} &= Y_n + k_n \sum_{i=1}^\nu \left( w_i f(\tilde{Y}_i) + \tilde{w}_i g(\tilde{Y}_i) \right). \label{eq:imex_rk}
	\end{align}
	
	Since the $A$ matrix defines an explicit scheme, it must be strictly lower triangular.  The implicit schemes we consider are DIRK schemes, therefore the matrix $B$ is lower triangular.  This condition ensures that $f(y)$ is always evaluated explicitly.  Butcher tableaus for some IMEX RK schemes are given in \S \ref{sec:num_examples}.


In the analytical development that follows we assume that the ODE systems of interest represent discretizations of PDE systems that have sufficient physical dissipation to develop stable discretizations for sufficiently fine meshes.  In this context the numerical results that are presented to confirm the analysis are associated with linear and nonlinear convection diffusion type systems, a nonlinear Burgers' equation with dissipation, and a coupled system from magnetohydrodynamics (MHD). 
The linearized form of these problems can be considered to be special cases of general convection-diffusion-reaction type systems.
References  \cite{HuRu07,PR2000} provide stability results for these types of systems integrated at the time-step of interest with IMEX type methods.
Adjoint analysis of systems which are convection dominated or have no physical dissipation and can support discontinuities is an active area of research and is beyond the scope of this study (see e.g. \cite{Giles:2010:CLA:1958393.1958407,Giles:2010:CLA:1958393.1958408}).

	\subsection{Finite Element Method}
		Adjoint-based error estimation requires that an approximate solution must be defined for the entire domain $[0,T]$.  This is done by using a finite element method to obtain the approximation.  
		 The same grid defined above is used to define the space of continuous piecewise polynomials,
		\begin{align}
			\Ccal^q = \{ w \in C^0([0,T];\Rbb^m) : w|_{I_n} \in \Pcal^q(I_n), 1 \leq n \leq N \},	
		\end{align}
		where $\Pcal^q(I_n)$ is the space of all polynomials of degree $q$ or less on $I_n$.  The continuous Galerkin finite element method of order $q+1$ for \eqref{eq:model_ode}, denoted cG(q), is defined interval-wise by,
		
		Find $Y \in \Ccal^q$ such that $Y(0) = y_0$ and for $n=0, \ldots, N-1$,
		\begin{align}
				\< \dot{Y},v_n\>_{I_n} =  \<f(Y) + g(Y), v_n \>_{I_n} , \quad \forall \; v_n \in \Pcal^{q-1}(I_n),
		\end{align}
		where $\< \cdot, \cdot \>_{[a,b]}  = \int_a^b ( \cdot, \cdot ) \, dt$ denotes the $L^2([a,b])$ inner product. Note that $Y \in \Pcal^q(I_n)$ and hence its time derivative, $\dot{Y}$ is well defined and easily obtained on each interval $I_n$.

	\subsection{Equivalent Finite Element Method} \label{ssec:equiv}
		In this section we construct a finite element method that is nodally equivalent to a particular IMEX scheme.  Nodal equivalence was developed in \cite{jdt13} to compute error estimates for explicit Runge-Kutta and Adams-Bashforth schemes and in \cite{CEG+2015} for multi-step IMEX schemes.  Two approximations are nodally equivalent if they are equal at the nodes $t_n$ of the discretization.  Two methods are nodally equivalent if they produce nodally equivalent approximations.  Therefore the finite element approximation constructed in this section is a function $Y(t) \in \Ccal^q$ with the property,
		\[ Y(t_n) = Y_n,\]
		where $Y_n$ is defined by an  IMEX  RK scheme \eqref{eq:imex_rk}. 
		
		
		We develop a nodally equivalent finite element method for the generic IMEX scheme defined by the Butcher tableaus \eqref{eq:general:Butcher}.  To obtain equivalency, we impose further conditions on these schemes by requiring all the elements of $\mbfd$ be distinct. 
		 The reason for this restriction is discussed in Remark~\ref{rem:1}.  
		
		We begin by defining an approximation operator $\Ical : H^1([0,T];\Rbb^m)\-> L^2([0,T];\Rbb^m)$.  Denoting the restriction by $\Ical^nY = \Ical Y|_{I_n}$, the operator is defined by,
	\begin{align}
		\Ical^n Y(t) = \sum_{i=1}^\nu \tilde{Y}_i \prod_{j=1,j \neq i}^\nu \frac{(t-t_{n+d_j})}{(t_{n+d_i}-t_{n+d_j})}. \label{eq:Ical}
	\end{align}
		where $\tilde{Y}_i$ are the stage variables for the IMEX scheme.  This operator approximates $Y$ by interpolating through the stage variables from the IMEX scheme.  
		Using this, the equivalent finite element method is defined by:
		
		Find $Y \in \Ccal^q$ such that $Y(0) = y_0$ and for $n=0,\ldots,N-1$,
		\begin{align}
			\< \dot{Y} , v_n\>_{I_n} = \< f(\Ical Y),v_n\>_{I_n,Q^f} + \< g(\Ical Y),v_n\>_{I_n,Q^g} \quad  \forall \; v_n \in \Pcal^{q-1}(I_n), \label{eq:equiv:FEM}
		\end{align}
		where the particular quadratures are defined by
		\begin{align}
			\<\f \>_{I_n,Q^f} &= k_n\sum_{i=1}^\nu w_i \f(t_{n+d_i}), \\
			\< \f \>_{I_n,Q^g} &= k_n\sum_{i=1}^\nu \tilde{w}_i \f(t_{n+d_i}).
		\end{align}
		\begin{rem}
		\label{rem:1}
			We now  see the reason for the condition that the elements of $\mbfd$ be distinct.  If they were not distinct, then $\Ical^nY$ would be ill-defined at $t_{n+d_i}$.  It is possible to develop a nodally equivalent method without this restriction, but the definition would be considerably more complicated.  We note that since the operator $\Ical$ does not depend on the vector $\mbfc$ there is no similar restriction on the vector $\mbfc$.  This is particular to autonomous systems, and would not be true if $f$ or $g$ had explicit dependence on $t$.
		\end{rem}

		Now  we show that the FEM approximation in \eqref{eq:equiv:FEM} is nodally equivalent to its corresponding IMEX scheme approximation.
		\begin{thm}
		\label{thm:equiv_fem_imex}
			The approximation $Y(t)$ obtained from the finite element method \eqref{eq:equiv:FEM} is nodally equivalent to the approximation $\{Y_n\}$ obtained from the IMEX scheme defined by the Butcher tableaus \eqref{eq:general:Butcher}.
		\end{thm}
		
		\begin{proof}
				We begin by setting $v_n=1$ in \eqref{eq:equiv:FEM}, evaluate the left side $\< \dot{Y} , 1 \>$ and rearrange to get,
				\begin{align}
					Y_{n+1} = Y_n + \< f(\Ical Y),1\>_{I_n,Q^f} + \< g(\Ical Y),1\>_{I_n,Q^g}.
				\end{align}
				Now applying the quadrature rules and \eqref{eq:Ical} which gives that $\Ical^nY(t_{n+d_i}) = \tilde{Y}_i$,
				\begin{align}
						Y_{n+1} &= Y_n + k_n \sum_{i=1}^\nu w_i f(\Ical^n Y(t_{n+d_i})) + k_n \sum_{i=1}^\nu \tilde{w}_i g(\Ical^n Y(t_{n+d_i})) \\
						&=Y_n + k_n \sum_{i=1}^\nu w_i f(\tilde{Y}_i) + k_n \sum_{i=1}^\nu \tilde{w}_i g(\tilde{Y}_i). \notag
				\end{align}
				which is the update formula for the IMEX scheme.  
		\end{proof}

	\subsection{Convergence of Equivalent Finite Element Method}
		We now prove that the above finite element method converges to the exact solution as the mesh is refined.  We consider second order or higher methods, as IMEX schemes of interest are generally at least second order.  The above equivalency shows that $Y(t)$ interpolates the exact solution $y(t)$ at the nodes $t_n$ to at least second order for such an IMEX scheme, and therefore for sufficiently small $k$ and smooth $y(t)$,
		\begin{align}
			\|y(t_n) - Y(t_n) \| \leq C k^2	,
		\end{align}
		where $k = \max_n k_n$ and $\| \cdot \|$ is the $\Rbb^m$ norm.
		
		The following lemma  discusses stability of interpolation to perturbations in the interpolated values \cite{quarteroni:2000}.
		\begin{lem}\label{lem:interp:stab}
			Let $p(t)$ and $\tilde{p}(t)$ interpolate the points $(t_0,z_0),\ldots, (t_r,z_r)$ and $(t_0,\tilde{z}_0), \ldots, (t_r,\tilde{z}_r)$ respectively with Lagrange polynomial interpolation.  If
			\begin{align}
				\max_{t\in[t_0,t_r]} \sum_{i=0}^r | \ell_i(t)| = \Lambda < \infty,
			\end{align}
			where $\ell_i(t)$ are the Lagrange basis functions then
			\begin{align}
				\|p - \tilde{p} \|_{\infty,[t_0,t_r]} \leq \Lambda \|\mbfz - \tilde{\mbfz} \|_{\infty},
			\end{align}
			where $\| \cdot \|_{\infty,[t_0, t_r]}$ is the $L^\infty([t_0,t_r])$ norm, $\| \cdot \|_\infty$ is the max norm, $\mbfz = [z_0,\ldots,z_r]^T$ and $\tilde{\mbfz} = [\tilde{z}_0,\ldots,\tilde{z}_r]^T$
		\end{lem}
		
		We now show convergence using a cG(1) method.
		\begin{thm} \label{thm:convergence}
			If $Y(t)$ is a cG(1) solution of \eqref{eq:equiv:FEM} corresponding to an IMEX scheme of at least second order, then for sufficiently small $k$,
			\begin{align}
				\|y(t) - Y(t) \|_{\infty} \leq (C_1+C) k^2,
			\end{align}
			where $C$ and $C_1$ are constants independent of $k$.
		\end{thm}
		\begin{proof}
			By equivalence shown above we have that
			\begin{align}
				\|y(t_n) - Y(t_n) \| \leq C k^2 \qquad \forall \; n = 0,\ldots,N-1.	
			\end{align}
			Let $I(t)$ be the continuous piecewise linear interpolant through the points $(t_n,y(t_n))$.  By definition, the restriction of $Y(t)$ to $I_n$ linearly interpolates the values $Y_n$ and $Y_{n+1}$.  Using  interpolation theory \cite{Brenner:2008} and Lemma \ref{lem:interp:stab} we have,
			\begin{align}
				\| y - Y \|_{\infty,[0,T]} &= \max_n \| y-Y\|_{\infty,I_n} \\
				&\leq \max_n \| y-I\|_{\infty,I_n} + \max_n \| I-Y\|_{\infty,I_n} \notag \\
				& \leq C_1k^2 + \max_n \Big\{ \max \big( | y(t_n) - Y(t_n) |, |y(t_{n+1}) - Y(t_{n+1})| \big) \Big \} \notag \\
				&\leq (C_1+C)k^2, \notag
			\end{align}
		where the constant $\Lambda = 1$ from Lemma \ref{lem:interp:stab},  since the linear Lagrange basis functions are positive and sum to unity, and $C_1$ is a constant from standard interpolation theory.  Therefore due to the convergence of the IMEX scheme, the finite element solution also converges.
		\end{proof}

\section{A posteriori Analysis}
\label{sec:apost_ana}
In this section, we derive a posteriori error estimates based on adjoint operators. In particular,  we derive estimates for the quantity $(y(T) - Y(T),\psi)$ where the vector $\psi$ specifies the QoI, see \eqref{eq:qoi}. 
The a posteriori analysis follows from the equivalence of the Runge-Kutta IMEX  scheme \eqref{eq:imex_rk} with the finite element method \eqref{eq:equiv:FEM}. Throughout this section, $y$ denotes the solution to the continuous ODE problem \eqref{eq:model_ode} whereas $Y$ represents the solution to the finite element method in \eqref{eq:equiv:FEM}. Theorem~\ref{thm:equiv_fem_imex} and \ref{thm:convergence} ensure the  error analysis of the finite element solution also applies to the Runge-Kutta IMEX  solution.

\subsection{Adjoint Problem}
There is no unique definition for adjoint operators corresponding to nonlinear operators. We employ  a definition which is  standard for a posteriori analysis~\cite{Estep:Larson:00}.  We consider the linearized operator,
\begin{equation}
\overline{H_{y,Y}}  = \int_0^1  \frac{d f (z)}{d y} + \frac{d g(z)}{d y} \, ds.
\end{equation}
where $z = sy + (1-s)Y$.  By the chain rule and the Fundamental Theorem of Calculus, this definition implies,
\begin{equation}
\label{eq:H_lin_action}
\begin{aligned}
\overline{H_{y,Y}}  (y - Y) &= \int_0^1  \frac{d f (z)}{d y} (y-Y)+ \frac{d g(z)}{d y}  (y-Y) \, d s\\
& = \int_0^1 \frac{d f(z)}{d s} + \frac{d g(z)}{d s} \, ds = \left ( f(y) - f(Y) \right) + \left(g(y) - g(Y) \right).
\end{aligned}
\end{equation}
The linearized operator is used to define the adjoint problem for the quantity-of-interest,
\begin{equation}
\label{eq:adj}
\begin{cases}
-\dot{\phi} = \overline{H_{y,Y}}^\top \phi, \qquad t  \in (T,0],\\
\phi(T) = \psi.
\end{cases}
\end{equation}
 Notice that the
adjoint problem is solved backwards in time.

\subsection{Error Representations}

Let $e = y - Y$ be the error. We employ the notation $\xi_n$ to denote the value at time $t_{n}$ for some function $\xi$.
\begin{lem}[Error Representation on an Interval]
On each interval $I_n$ we have,
\begin{equation}
\label{eq:one_interval}
(e_{n+1}  ,  \phi_{n+1}) = (e_{n}  ,  \phi_{n}) + \langle f(Y) + g(Y) - \dot{Y}, \phi \rangle_{I_n}.
\end{equation}
\end{lem}

\begin{proof}
The proof is standard, e.g. see Section 8.1 in \cite{eehj_actanum_95}.  For completeness we give it here.  We begin by taking the $L^2(I_n)$ inner product of $e$ with the differential equation \eqref{eq:adj}.

\begin{align}
    0 &= \< e, -\dot{\f} - \overline{H_{y,Y}}^\top \f \>_{I_n} \notag \\
    &= \< e, -\dot{\f} \>_{I_n} - \< \overline{H_{y,Y}}^\top \f \>_{I_n} \notag \\
    &= \<\dot{e}, \f \>_{I_n} + (e_n, \f_n) - (e_{n+1}, \f_{n+1}) - \< \overline{H_{y,Y}}^\top \f \>_{I_n}. \notag
\end{align}
Now using \eqref{eq:H_lin_action} and \eqref{eq:model_ode} we obtain
\begin{align*}
    0 &=\<\dot{e}, \f \>_{I_n} + (e_n, \f_n) - (e_{n+1}, \f_{n+1}) \\
     &\qquad + \<f(Y) + g(Y), \f \>_{I_n} - \< f(y) + g(y), \f \>_{I_n} \\
     &= (e_n, \f_n) - (e_{n+1}, \f_{n+1}) + \<f(Y) + g(Y) - \dot{Y}, \f \>_{I_n}.
\end{align*}
Rearranging we obtain \eqref{eq:one_interval}.
\end{proof}

\begin{thm}[Error Representation]
If $y(t)$ and $Y(t)$ are solutions of \eqref{eq:model_ode} and \eqref{eq:equiv:FEM} respectively, and $\f$ is a solution of the adjoint problem \eqref{eq:adj}, then the error in the quantity-of-interest defined by $\psi$ is given by,
\label{thm:err_rep}
\begin{equation}
\label{err:imex_1}
{Q}(y - Y) = (y(T) - Y(T)  ,  \psi) =E1 + E2 + E3,
\end{equation}
where,
\begin{equation}
E1 = \sum_{n=0}^{N-1} E1_n, \qquad E2 = \sum_{n=0}^{N-1} E2_n, \qquad E3 = \sum_{n=0}^{N-1} E3_n,
\end{equation}
and
\begin{equation}
\begin{aligned}
&E1_n =    \langle - \dot{Y}, \phi  - \pi_n \phi \rangle_{I_n} + \langle f(\mathcal{I} Y) , \phi - \pi_n \phi \rangle_{{I_n},Q^f} +  \langle g(\mathcal{I} Y)  , \phi- \pi_n \phi \rangle_{{I_n},Q^g},\\
&E2_n =    \langle f(Y) , \phi \rangle_{I_n} - \langle f(\mathcal{I} Y) , \phi \rangle_{{I_n},Q^f}, \\
&E3_n =     \langle g(Y), \phi \rangle_{I_n} -   \langle g(\mathcal{I} Y)  , \phi \rangle_{{I_n},Q^g}.
\end{aligned}
\end{equation}
Here $\pi_n \phi$ represents a projection of $\phi\lvert_{I_n}$ onto
the space $\mathcal{P}^{q-1}(I_n)$. The terms $E1$, $E2$ and $E3$ represent the discretization, explicit and implicit contributions to the error respectively.
\end{thm}

\begin{proof}
Adding \eqref{eq:one_interval} over all intervals for $n = 0, \ldots, N-1$ we have,
\begin{equation}
\label{eq:summed_err}
\sum_{n=0}^{N-1}(e_{n+1}  ,  \phi_{n+1}) = \sum_{n=0}^{N-1}\left[(e_{n}  ,  \phi_{n}) + \langle f(Y) + g(Y) - \dot{Y}, \phi \rangle_{I_n} \right].
\end{equation}
Now $\sum_{n=0}^{N-1}(e_{n+1}  ,  \phi_{n+1})  = \sum_{n=1}^{N-1} (e_{n}  ,  \phi_{n}) + (e_{N}  ,  \phi_{N}) $. Using this in \eqref{eq:summed_err} along with the fact that the numerical solution satisfies the initial condition, i.e. $e_0 = 0$, we arrive at,
\begin{equation}
\label{eq:fem_err_adj_a}
 (e_{N}  ,  \phi_{N})  = \sum_{n=0}^{N-1} \left[ \langle f(Y) + g(Y) - \dot{Y}, \phi \rangle_{I_n} \right].
\end{equation}
Adding and subtracting $\sum_{n=0}^{N-1} \langle f(\mathcal{I} Y) , \phi \rangle_{{I_n},Q^f} $ and $\sum_{n=0}^{N-1}\langle g(\mathcal{I} Y)  , \phi \rangle_{{I_n},Q^g}$ to the right hand side of \eqref{eq:fem_err_adj_a}  leads to,
\begin{equation}
\label{eq:fem_err_adj_b}
 (e_{N}  ,  \phi_{N})  = \sum_{n=0}^{N-1} \left[  \langle - \dot{Y}, \phi   \rangle_{I_n} + \langle f(\mathcal{I} Y) , \phi  \rangle_{{I_n},Q^f} +  \langle g(\mathcal{I} Y)  , \phi  \rangle_{{I_n},Q^g}  + E2_n + E3_n \right].
\end{equation}
Now, since $\pi_n \phi \in \mathcal{P}^{q-1}(I_n)$, we have from \eqref{eq:equiv:FEM} for $n = 0, \ldots, N-1$,
\begin{equation}
\label{eq:galerkin_ortho}
\langle  \dot{Y}, \pi_n \phi  \rangle_{I_n} - \langle f(\mathcal{I} Y) , \pi_n \phi  \rangle_{{I_n},Q^f}  - \langle g(\mathcal{I} Y)  , \pi_n \phi \rangle_{{I_n},Q^g} = 0, 
\end{equation}
Combining \eqref{eq:fem_err_adj_b} and \eqref{eq:galerkin_ortho} completes the proof.
\end{proof}

\subsection{A posteriori analysis for time dependent quantities-of-interest}
The QoI defined in \eqref{eq:qoi} is based only on the final time of the solution. The analysis above is easily modified for other quantities of interest. For example consider the time dependent QoI,
\begin{equation}
\tilde{Q}(y) \equiv \int_{0}^{T} (y(t), \tilde{\psi}(t) ) \, dt,
\end{equation}
 where the time dependent function, $\tilde{\psi}: \Rbb \rightarrow \Rbb^m$, specifies the QoI. Define the adjoint problem as,
\begin{equation}
\label{eq:adj_2}
\begin{cases}
-\dot{\tilde{\phi}} = \overline{H_{y,Y}}^\top \tilde{\phi} + \tilde{\psi}, \qquad t  \in (T,0]\\
\tilde{\phi}(T) = 0.
\end{cases}
\end{equation}
Note that this adjoint problem differs from \eqref{eq:adj} in the
initial conditions and data on the right hand side. This leads to the
following error representation.
\begin{thm}
If $y(t)$ and $Y(t)$ are solutions of \eqref{eq:model_ode} and \eqref{eq:equiv:FEM} respectively, and $\tilde{\phi}$ is a solution of the adjoint problem \eqref{eq:adj_2}, then the error in the quantity-of-interest specified by $\tilde{\psi}$ is given by,
\begin{equation}
\label{err:imex_1_diff_qoi}
\tilde{Q}(y - Y) = \int_0^T(y(t) - Y(t)  ,  \tilde{\psi}(t)) =\tilde{E1} + \tilde{E2} + \tilde{E3},
\end{equation}
where,
\begin{equation}
\begin{aligned}
&\tilde{E1} =   \sum_{n=0}^{N-1} \langle - \dot{Y}, \tilde{\phi}  - \pi_n \tilde{\phi} \rangle_{I_n} + \langle f(\mathcal{I} Y) , \tilde{\phi} - \pi_n \tilde{\phi} \rangle_{{I_n},Q^f} +  \langle g(\mathcal{I} Y)  , \tilde{\phi}- \pi_n \tilde{\phi} \rangle_{{I_n},Q^g},\\
&\tilde{E2} =  \sum_{n=0}^{N-1}  \langle f(Y) , \tilde{\phi} \rangle_{I_n} - \langle f(\mathcal{I} Y) , \tilde{\phi} \rangle_{{I_n},Q^f}, \\
&\tilde{E3} =  \sum_{n=0}^{N-1}   \langle g(Y), \tilde{\phi} \rangle_{I_n} -   \langle g(\mathcal{I} Y)  , \tilde{\phi} \rangle_{{I_n},Q^g}.
\end{aligned}
\end{equation}
\end{thm}
\begin{proof}
    The proof is similar to that of Theorem~\ref{thm:err_rep}.
\end{proof}

\subsection{Extension of the analysis for space-time discretization of
  PDEs}
The primary aim of this article is to quantify the error due to time
integration using the IMEX Runge-Kutta schemes. Hence the analysis presented in this article deals with the error in the
numerical solution of an ODE system. As mentioned earlier, such ODE
systems often arise from spatial discretization of partial
differential equations. 
 In this article, we ignore the error in the solution to the PDE due
 to this spatial discretization. The extension of the analysis to the
 case of PDEs to quantify the effect of spatial discretization follows directly from the analysis of ODEs, e.g. see
 \cite{CEG+2015, Collins:2014} for details.

\section{Numerical Experiments}
\label{sec:num_examples}

\subsection{Algorithmic Details}
In this section we consider numerical examples for systems of ODEs of
the form \eqref{eq:model_ode} arising from the spatial discretization
of PDEs. The spatial derivatives are discretized  using a second order
central finite difference scheme in space. The spatial discretization
parameter is referred to as $h$. 

In order to estimate the error with the error representation formulas, we must,
\begin{itemize}
    \item Solve the forward problem with an IMEX scheme to obtain the solution $\{Y_n\}$.
    \item Determine the equivalent finite element solution $Y(t)$.
    \item Solve the associated adjoint problem using the finite element solution $Y(t)$ in the operator $\overline{H_{y,Y}}$.
\end{itemize}
Let us assume we are using an IMEX scheme of order $p$.  To find the equivalent finite element solution we chose an order $q$ for the finite element method.  This order is chosen to correspond to the IMEX scheme by setting $q = p-1$.  The intermediate values of the solution on each subinterval $I_n$ are determined by solving a simple mass-matrix linear system.

Finally, the adjoint equation is solved.  In theory,
the adjoint is obtained by linearizing around a combination of the
discrete solution and the true solution. In practice, we linearize
around the discrete solution only~\cite{Estep:Larson:00}. Moreover,
the adjoint solution needs to be approximated numerically. We
approximate the adjoint solution  using the cG($q$+1) finite element
method. The cG($q$+1) finite element
method for the adjoint equation \eqref{eq:adj} is  defined interval-wise by,
		
		Find $\Phi \in \Ccal^{q+1}$ such that $\Phi(T) = \psi$ and for $n=N-1, \ldots, 0$,
		\begin{align}
				\< -\dot{\Phi},v_n\>_{I_n} =  \<\overline{\tilde{H}_{y,Y}}^\top\Phi, v_n \>_{I_n} , \quad \forall \; v_n \in \Pcal^{q}(I_n),
		\end{align}
where $\overline{\tilde{H}_{y,Y}}$ is obtained by substituting $y = Y$ in the expression for $\overline{{H}_{y,Y}}$ ~\cite{Estep:Larson:00}. Notice that the adjoint problem is solved backwards in time. That is, the initial conditions are posed at time $t = T$, which corresponds to the final time for the original ODE~\eqref{eq:model_ode}. The adjoint solution is computed interval by interval by starting at the  interval $I_{N-1}$, then proceeding to $I_{N-2}$ and so on until the interval $I_0$ is reached.  Given the nodally equivalent finite element solution $Y(t)$, the computation of $\overline{\tilde{H}_{y,Y}}$ on any interval $I_n$ is straightforward.
The higher order approximation of the adjoint
problem ensures that the error estimates are accurate.

\subsection{Examples}
In this paper we examine three IMEX schemes in particular, IMEX Midpoint(1,2,2), IMEX-SSP3(3,3,2) and IMEX-SSP3(4,3,3).  The names of these schemes are standard and use a triplet notation $(s, \s, p)$ where $s$ is the number of stages in the implicit method, $\s$ is the number of stages in the explicit method, and $p$ is the order of the method as a whole.  The Butcher tableaus for these schemes are shown in Tables \ref{tab:midpoint}, \ref{tab:SSP3:2} and \ref{tab:SSP3:3}. The methods considered here consist of an A-stable IMEX integrator (IMEX Midpoint(1,2,2)), and two IMEX integrators of second- and third-order that have strong-stability-preserving properties for the 
explicit operators and a L-stable property for the implicit integrator (IMEX-SSP3(3,3,2) and IMEX-SSP3(4,3,3)). Here the IMEX midpoint  represents a method of interest for parabolic
problems with smooth solutions and sufficient levels of physical dissipation \cite{ARS97}. The SSP L-stable methods represent methods of interest for systems with hyperbolic character or parabolic behavior with limited dissipation. In this context numerical solutions with poorly resolved gradients and/or discontinuities  can result \cite{PR2000,IMEXRKHyperbolicStiffRelax2005, ARS97,HuRu07}. 

		\begin{table}[h]
		\begin{center}
			\begin{tabular}{l|ll}
			0   & 0   & 0 \\
			1/2 & 1/2 & 0 \\ \hline
			    & 0   & 1
			\end{tabular}
		\hspace{.5in}
			\begin{tabular}{l|ll}
			0   & 0 & 0   \\
			1/2 & 0 & 1/2 \\ \hline
			    & 0 & 1  
			\end{tabular}
			\caption{Butcher Tableau for the explicit(left) and implicit(right) portion of the IMEX scheme Midpoint(1,2,2).}
			\label{tab:midpoint}
		\end{center}
		\end{table}

		\begin{table}[h!]
		\begin{center}
			\begin{tabular}{l|lll}
			0   & 0   & 0   & 0   \\
			1   & 1   & 0   & 0   \\
			1/2 & 1/4 & 1/4 & 0   \\ \hline
			    & 1/6 & 1/6 & 2/3
			\end{tabular}
		\hspace{.5in}
			\begin{tabular}{l|lll}
			$\g$   & $\g$       & 0    & 0    \\
			1-$\g$ & 1-$2\g$    & $\g$ & 0    \\
			1/2    & 1/2 - $\g$ & 0    & $\g$ \\ \hline
			       & 1/6        & 1/6  & 2/3 
			\end{tabular}
			\caption{Butcher Tableau for the explicit(left) and implicit(right) portion of IMEX-SSP3(3,3,2). $\g = 1- \frac{1}{\sqrt{2}}$}
			\label{tab:SSP3:2}
		\end{center}
		\end{table}	

		\begin{table}[h!]
		\begin{center}
			\begin{tabular}{c|cccc}
			0   & 0   & 0   & 0 & 0  \\
			0   & 0   & 0   & 0 & 0  \\
			1   & 0   & 1   & 0 & 0  \\
			1/2 & 0 & 1/4 & 1/4  & 0  \\ \hline
			      & 0 & 1/6 & 1/6 & 2/3
			\end{tabular}
		\hspace{.15in}
			\begin{tabular}{c|cccc}
			$\a$   & $\a$   & 0   & 0 & 0  \\
			0   & $-\a$   & $\a$   & 0 & 0  \\
			1   & 0   & $1-\a$   & $\a$ & 0  \\
			1/2 & $\b$ & $\eta$ & $1/2 - \b - \eta - \a$  & $\a$  \\ \hline
			      & 0 & 1/6 & 1/6 & 2/3
			\end{tabular}
			\caption{Butcher Tableau for the explicit(left) and implicit(right) portion of IMEX-SSP3(4,3,3).
			$\a = 0.24169426078821, \b = 0.06042356519705, \eta = 0.12915286960590$}
			\label{tab:SSP3:3}
		\end{center}
		\end{table}

The approximation of the adjoint solution  leads to an ``error
estimate'' from the error representation \eqref{err:imex_1}.
The effectivity ratio measures the accuracy of the estimate and is defined as,
\begin{equation*}
\rho_{\rm eff} = \frac{\mbox{Estimated error}}{\mbox{True error}} \, .
\end{equation*}
An accurate error estimate has an effectivity ratio close to one. In our examples, the true solution is unknown and is approximated to a high degree of accuracy using {\sc Matlab}'s ODE solver.

We present four examples.
The first two examples  arise from the finite difference discretization of scalar-valued PDEs which were previously considered in~\cite{ARS97}. In these examples we chose $f$ to be the term arising from the first-order spatial derivatives, while $g$ represents the term arising from second-order spatial derivatives. In the third example we illustrate how the choice of $f$ and $g$ effects the components of the error estimate. The final example is a simplified 1D Magneto-Hydrodynamics problem.

\subsubsection{Linear PDE} \label{sec:linear_pde_fd}
Consider the scalar valued linear PDE
\begin{equation}
\begin{cases}
\begin{gathered}
\begin{aligned}
&\dot{u}  + \sin(2 \pi x)u_x = \gamma u_{xx},\quad && (x,t) \in  [0,1]\times (0,T],\\
&u(x,0) = \sin(2 \pi x),\quad && x \in  [0,1] , \label{eq:lin:pde}
\end{aligned}
\end{gathered}
\end{cases}
\end{equation}
with periodic boundary conditions.

We choose the spatial discretization
parameter as $h = 1/40$ and the QoI as $\psi = [ \mathbf{1}  \quad \mathbf{0}]^\top$,
where  $\mathbf{1}$ is an  $(m/2+1)$ vector of all ones and
$\mathbf{0}$ is an $(m/2-1)$ vector of all zeros. QoIs of this form often arise from evaluating spatial integrals of the
PDE solution, $u(x,t)$, at the final time, i.e.  
$C \int_a^b u(x,t) \, dx$. 
\begin{table}[!ht]
\centering
\begin{tabular}{ l| c |c |c |c | c }
\toprule
Scheme & Comp. Err. &Eff. Ratio & $E1$ & $E2$ & $E3$\\
\midrule
Mid(1,2,2)   &-1.50E-06 &0.99 &1.96E-06 &5.01E-06 &-8.48E-06  \\
SSP3(3,3,2)  &6.11E-07  &1.00 & 1.61E-06 & 2.19E-06& -3.20E-06  \\
SSP3(4,3,3)  &7.89E-09  &1.14 &4.97E-09 &-2.56E-07 & 2.59E-07\\
\bottomrule
\end{tabular}
\caption{Results for  the problem in \S \ref{sec:linear_pde_fd} with the final time $T = 2.0$, $k_n = 1/40$, and $\gamma = 0.1$ using the Midpoint(1,2,2), IMEX-SSP3(3,3,2) and IMEX-SSP3(4,3,3)schemes.}
\label{tab:example_1_additional_1}
\end{table}
\begin{table}[!ht]
\centering
\begin{tabular}{ l| c |c |c |c | c }
\toprule
Scheme & Comp. Err. &Eff. Ratio & $E1$ & $E2$ & $E3$\\
\midrule
Mid(1,2,2)   &  -1.71E-05 & 0.99 & 4.036E-06 & -2.08E-05 &-3.37E-07              \\
SSP3(3,3,2)  &   1.08E-06 & 1.00 & 3.27E-06  & -2.17E-06 &-1.08E-08           \\
SSP3(4,3,3)  &    6.41E-07& 1.00 & 1.75E-08  & 5.37E-07  &8.72E-08          \\
\bottomrule
\end{tabular}
\caption{Results for  the problem in \S \ref{sec:linear_pde_fd} with the final time $T = 2.0$, $k_n = 1/40$, and $\gamma = 0.01$ using the Midpoint(1,2,2), IMEX-SSP3(3,3,2) and IMEX-SSP3(4,3,3)schemes.}
\label{tab:example_1_additional_2}
\end{table}

In the first set of numerical experiments, we set the time step as $k_n = 1/40$.  
We report results for 2 different levels of diffusion coefficient, $\gamma = 0.1~,0.01$, and evolve the solution to a final time of
$T = 2.0$. Results are shown in Tables~ \ref{tab:example_1_additional_1} and \ref{tab:example_1_additional_2}  respectively. The results indicate that the error estimate is quite
accurate for all schemes, as shown by the effectivity ratio.

 In the second set of experiments, we increase the time step to $k_n = 1/10$ and solve for two different final times of   $T=1.0$ and $T = 2.0$. 
The results in Tables~ \ref{tab:1} and \ref{tab:2} indicate that the error estimate is quite
accurate for all schemes, as shown by the effectivity ratio.  The error estimates are even accurate when the error is
quite large, as is seen in Table~\ref{tab:2} for the case of the
Midpoint(1,2,2) scheme. The large error in this scheme is due to instabilities that develop in the solution, see Figure \ref{fig:soln_lin_others_nu_0.1}.  The results
for $\gamma = 0.01$ are shown in Tables \ref{tab:3} and \ref{tab:4} for
$T = 1.0$ and $T = 2.0$ respectively. Once again the error estimates
are quite accurate. Moreover, the error estimates indicate instability
in the IMEX-SSP3(4,3,3) solution in addition to the Midpoint(1,2,2)
solution for $T=2.0$. This is also seen in the plots of the solutions
in Figure~\ref{fig:soln_lin_nu_0.01} where we observe that
IMEX-SSP3(3,3,2) remains stable whereas the other two schemes develop
instabilities after a certain time. The IMEX-SSP3(4,3,3)
seems to be more stable than the Midpoint(1,2,2) solution.

\begin{table}[!ht]
\centering
\begin{tabular}{ l| c |c |c |c | c }
\toprule
Scheme & Comp. Err. &Eff. Ratio & $E1$ & $E2$ & $E3$\\
\midrule
Mid(1,2,2)            &-4.53E-03  & 0.99  & -7.43E-02 &   4.26E-02 &  2.72E-02                             \\
SSP3(3,3,2)          &  1.428E-03&   0.99&   3.37E-03&   6.89E-03  &-8.83E-03                         \\
SSP3(4,3,3)         &  -5.83E-04 &  1.00 &  2.13E-04&  -4.58E-03 &  3.78E-03                         \\
\bottomrule
\end{tabular}
\caption{Results for  the problem in \S \ref{sec:linear_pde_fd} with the final time $T = 1.0$, $k_n = 1/10$, and $\gamma = 0.1$ using the Midpoint(1,2,2), IMEX-SSP3(3,3,2) and IMEX-SSP3(4,3,3)schemes.}
\label{tab:1}
\end{table}

\begin{table}[!ht]
\centering
\begin{tabular}{ l| c |c |c |c | c }
\toprule
Scheme & Comp. Err. &Eff. Ratio & $E1$ & $E2$ & $E3$\\
\midrule
Mid(1,2,2)   &  -1.90E+02   &1.00  &-1.07E+03   &1.45E+03  &-5.72E+02              \\
SSP3(3,3,2)  &     2.68E-06   &1.00  & 2.62E-05   &4.64E-05  &-6.99E-05           \\
SSP3(4,3,3)  &    -1.89E-06   &0.99 &-6.30E-05  &-1.91E-05  & 8.11E-05          \\
\bottomrule
\end{tabular}
\caption{Results for  the problem in \S \ref{sec:linear_pde_fd} with the final time $T = 2.0$, $k_n = 1/10$, and $\gamma = 0.1$ using the Midpoint(1,2,2), IMEX-SSP3(3,3,2) and IMEX-SSP3(4,3,3)schemes.}
\label{tab:2}
\end{table}

\begin{figure}[!ht]
\centering
\begin{subfigure}[!ht]{0.46\textwidth}
\centering
\includegraphics[width = \textwidth]{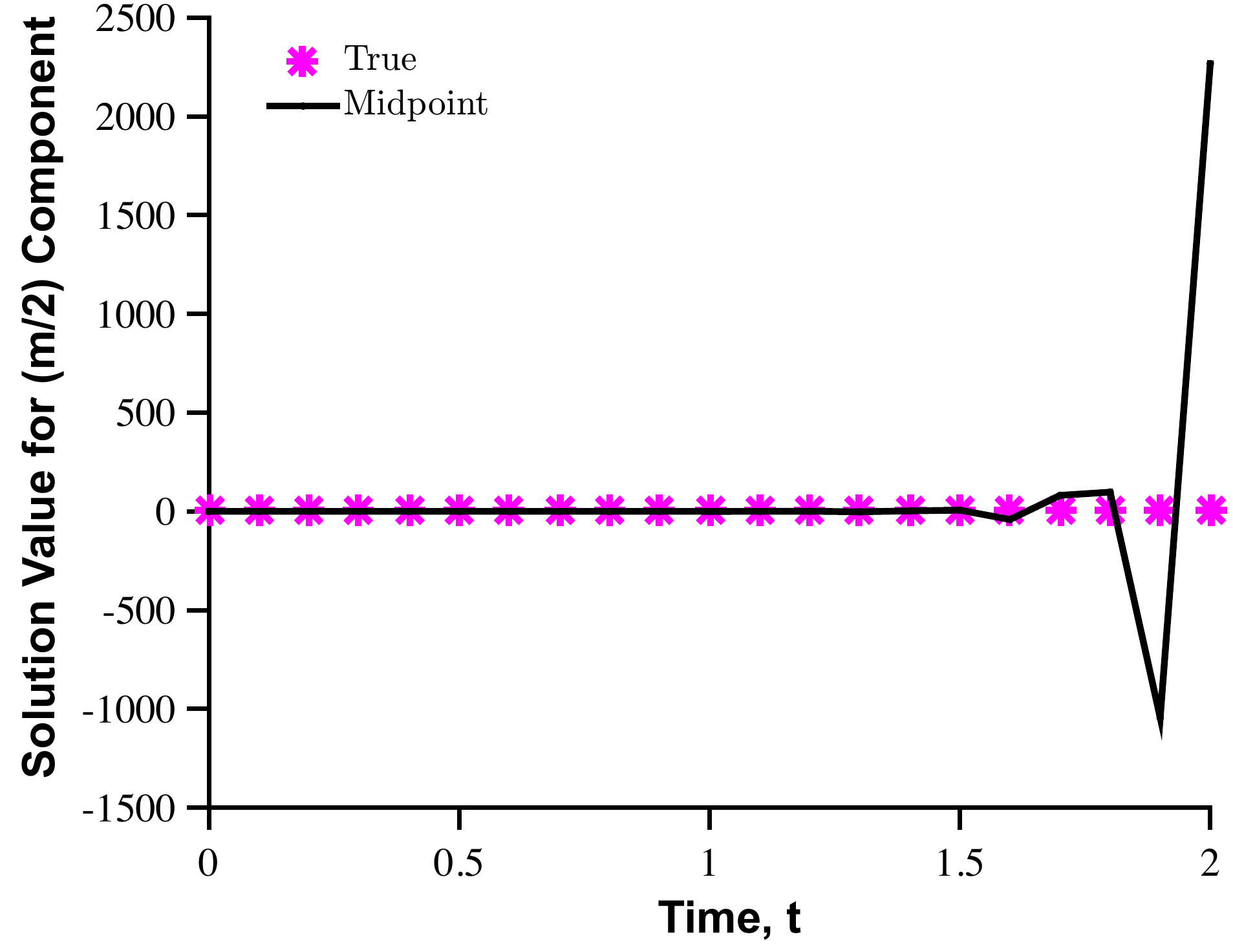} 
\caption{}
\label{fig:soln_lin_others_nu_0.1}
\end{subfigure}
\hfill
\begin{subfigure}[!ht]{0.46\textwidth}
\includegraphics[width = \textwidth]{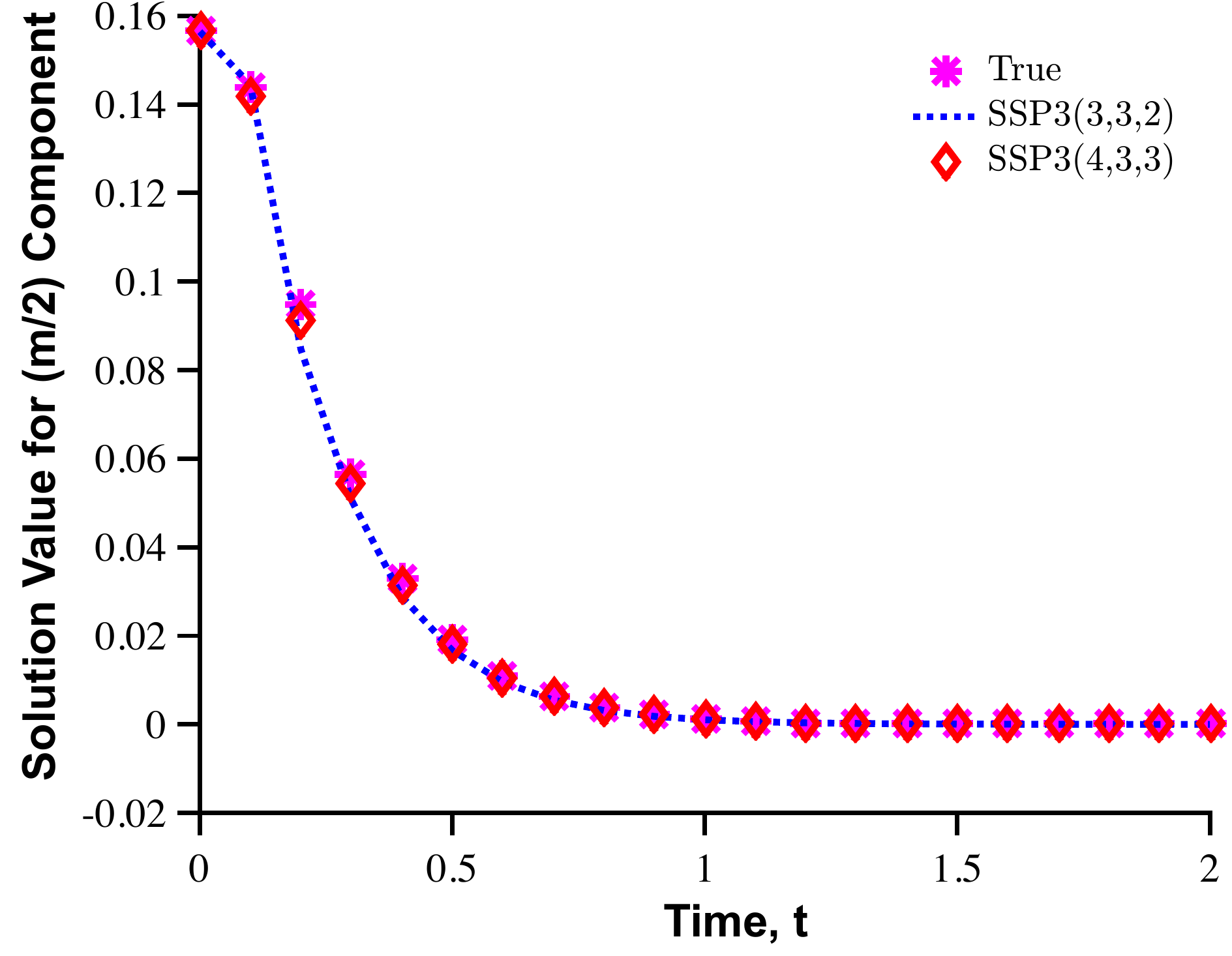} 
\caption{}
\label{fig:soln_lin_mp_nu_0.1}
\end{subfigure}
\caption{Plot of the value of $m/2$th component of the solutions (vertical axis) of \eqref{eq:lin:pde} for $\gamma = 0.1$ and $k_n = 1/10$. The horizontal axis denotes the time, $t$.}
\label{fig:soln_lin_nu_0.1}
\end{figure}

\begin{table}[!ht]
\centering
\begin{tabular}{ l| c |c |c |c | c }
\toprule
Scheme & Comp. Err. &Eff. Ratio & $E1$ & $E2$ & $E3$\\
\midrule
Mid(1,2,2)            &-1.83E-01           &0.99       &-2.27E-1      &9.34E-2  &-4.944E-2 \\
SSP3(3,3,2)          &6.36E-03    & 1.00           &1.12E-2  &-8.63E-3 &4 3.71E-3 \\
SSP3(4,3,3)         & 8.51E-04      &0.99        &-8.72E-3  &2.33E-3 &7.24E-3\\
\bottomrule
\end{tabular}
\caption{Results for  the problem in \S \ref{sec:linear_pde_fd} with the final time $T = 1.0$, $k_n = 1/10$, and $\gamma = 0.01$ using the Midpoint(1,2,2), IMEX-SSP3(3,3,2) and IMEX-SSP3(4,3,3)schemes.}
\label{tab:3}
\end{table}

\begin{table}[!ht]
\centering
\begin{tabular}{ l| c |c |c |c | c }
\toprule
Scheme & Comp. Err. &Eff. Ratio & $E1$ & $E2$ & $E3$\\
\midrule
Mid(1,2,2)   &3.22E+03   &1.00  &-1.41E+03   &3.83E+03   &8.00E+02 \\
SSP3(3,3,2)  &8.02E-03  & 1.00  &-1.96E-03   &8.31E-03   &1.67E-03 \\
SSP3(4,3,3)  &-1.46E+00  & 1.00 & -2.24E+00 & -3.28E-01  & 1.10E+00\\
\bottomrule
\end{tabular}
\caption{Results for  the problem in \S \ref{sec:linear_pde_fd} with the final time $T = 2.0$, $k_n = 1/10$, and $\gamma = 0.01$ using the Midpoint(1,2,2), IMEX-SSP3(3,3,2) and IMEX-SSP3(4,3,3)schemes.}
\label{tab:4}
\end{table}

\begin{figure}[!ht]
\centering
\begin{subfigure}[!ht]{0.48\textwidth}
\centering
\includegraphics[width = \textwidth]{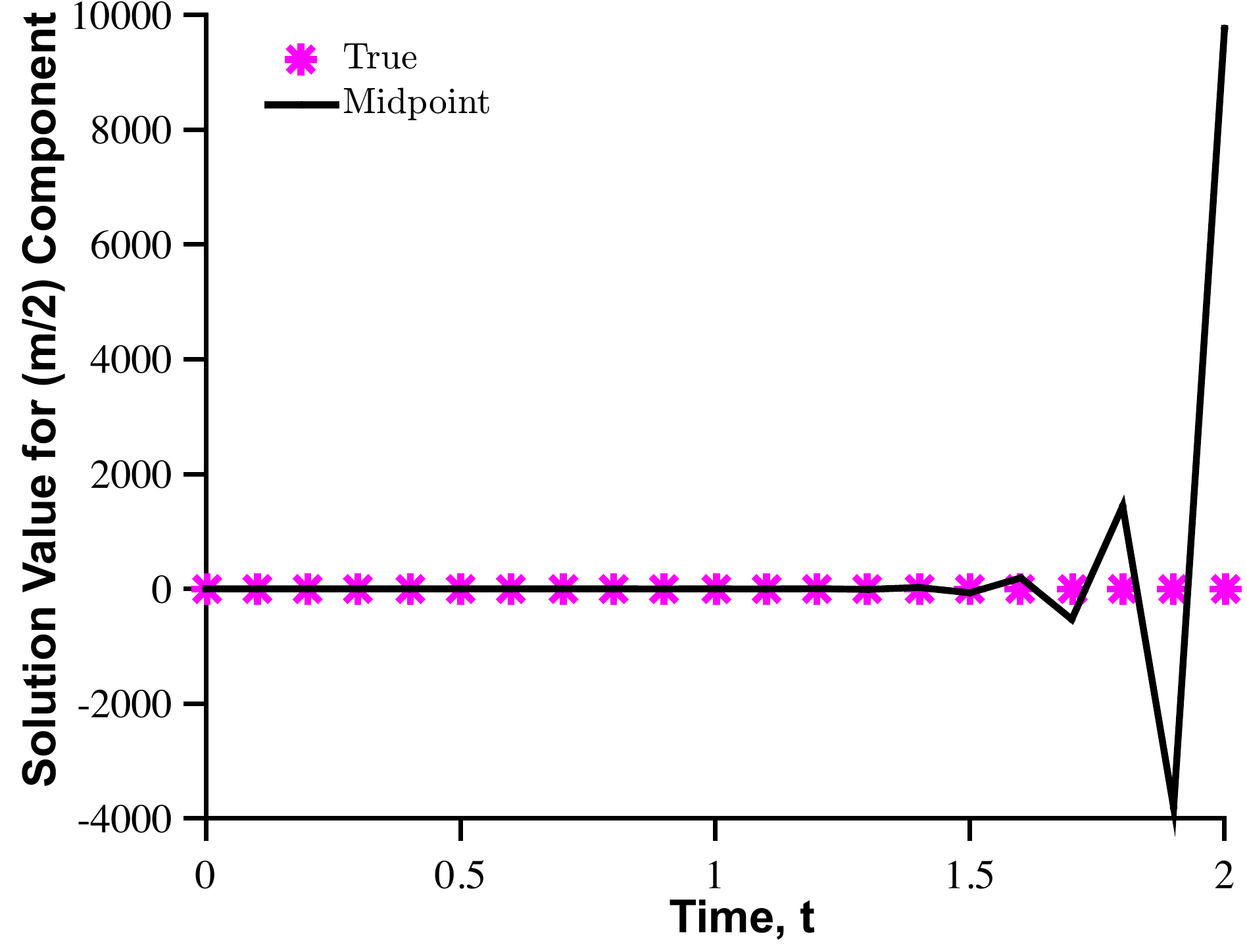} 
\caption{}
\label{fig:soln_lin_others_nu_0.01}
\end{subfigure}
\hfill
\begin{subfigure}[!ht]{0.48\textwidth}
\includegraphics[width = \textwidth]{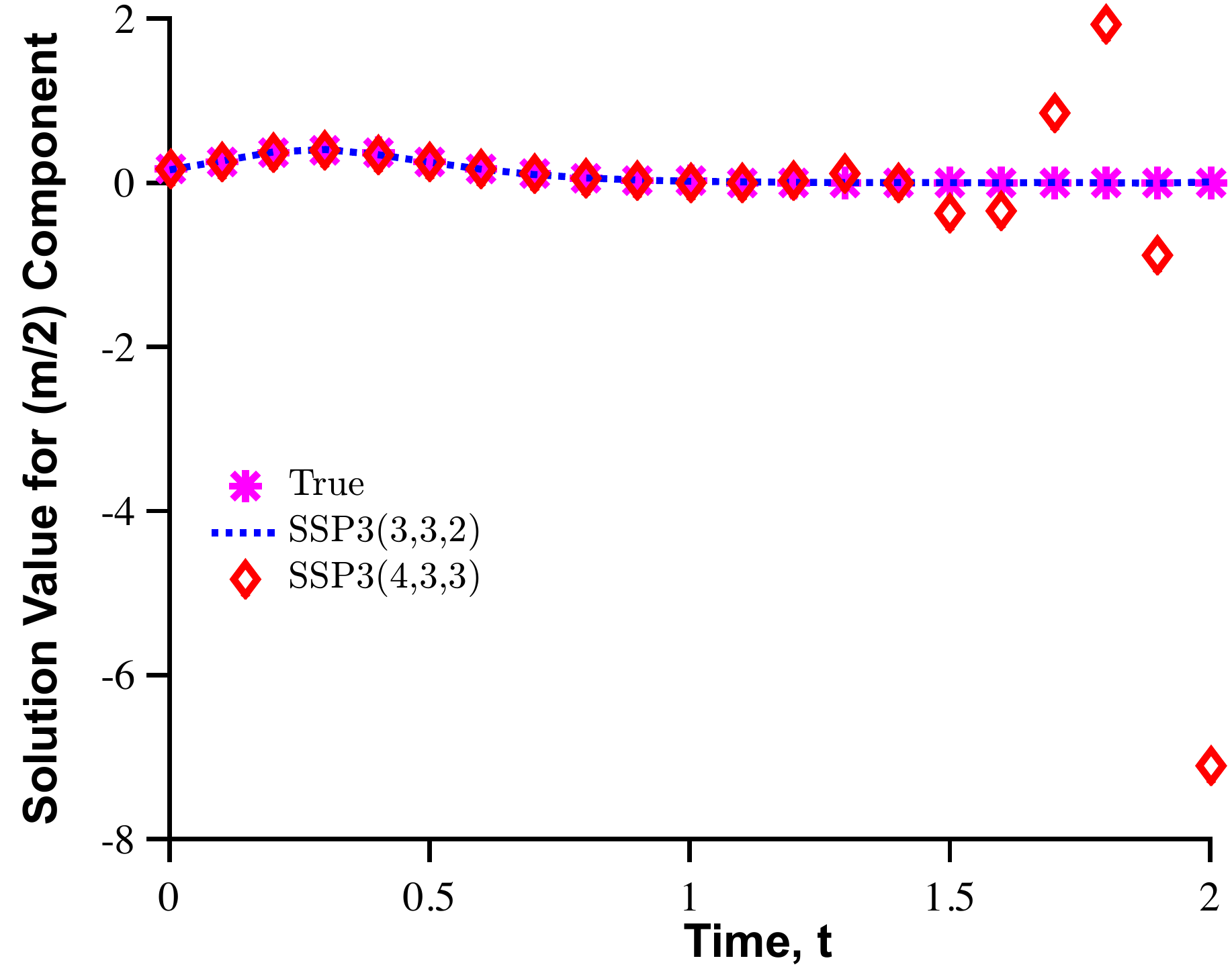} 
\caption{}
\label{fig:soln_lin_mp_nu_0.01}
\end{subfigure}
\caption{Plot of the value of $m/2$th component (vertical axis) of the solutions of \eqref{eq:lin:pde} for $\gamma = 0.01$ and $k_n = 1/10$. The horizontal axis denotes the time, $t$.}
\label{fig:soln_lin_nu_0.01}
\end{figure}

\subsubsection{Damped non-linear Burger's equation} \label{sec:burgers_fd}

The damped non-linear Burger's equation is
\begin{equation}
\label{eq:burgers}
\begin{cases}
\begin{gathered}
\begin{aligned}
&\dot{u} +  u u_x = \gamma u_{xx}, \quad && (x,t) \in  [-1,1]\times (0,T], \\
&u(x,0) = \sin( \pi x),\quad && x \in  [-1,1],
\end{aligned}
\end{gathered}
\end{cases}
\end{equation}
which we consider with periodic boundary conditions. The QoI is again chosen as $\psi = [ \mathbf{1}  \quad \mathbf{0}]^\top$,
where  $\mathbf{1}$ is an  $(m/2+1)$ vector of all ones and
$\mathbf{0}$ is an $(m/2-1)$ vector of all zeros. We choose the
spatial discretization parameter as $h = 1/40$, the time step as
$k_n = 1/20$ and $\gamma = 0.05$.  The results for two different values of the final time,
$T$, are shown in Tables~\ref{tab:7} and \ref{tab:8}. The error
estimate is again quite accurate, with effectivity ratios close to
one. The IMEX-SSP3(4,3,3)  method has the least error for this example.

  \begin{table}[!ht]
\centering
\begin{tabular}{ l| c |c |c |c | c }
\toprule
Scheme & Comp. Err. &Eff. Ratio & $E1$ & $E2$ & $E3$\\
\midrule
Mid(1,2,2)            &-8.13E-03   &1.00  &-1.80E-04  &-1.98E-02   &1.18E-02       \\
SSP3(3,3,2)         &  -6.84E-03   &1.00  &-1.64E-04  &-8.47E-03   &1.79E-03     \\
SSP3(4,3,3)         &  -2.30E-04   &1.00  &1.04E-05   &7.14E-04  &-9.55E-04      \\
\bottomrule
\end{tabular}
\caption{Results for  the problem in \S \ref{sec:burgers_fd} with the final time $T = 1.0$ using the Midpoint(1,2,2), IMEX-SSP3(3,3,2) and IMEX-SSP3(4,3,3)schemes..}
\label{tab:7}
\end{table}

  \begin{table}[!ht]
\centering
\begin{tabular}{ l| c |c |c |c | c }
\toprule
Scheme & Comp. Err. &Eff. Ratio & $E1$ & $E2$ & $E3$\\
\midrule
Mid(1,2,2)            & 1.10E-03   &0.98     &-7.17E-04 & -4.00E-03   &5.82E-03       \\
SSP3(3,3,2)         &  -1.16E-03  & 1.00 &-5.98E-04  &-1.51E-03   &9.50E-04     \\
SSP3(4,3,3)         &   1.35E-04   &0.99   &4.20E-06   &3.59E-04    &-2.28E-04      \\
\bottomrule
\end{tabular}
\caption{Results for  the problem in \S \ref{sec:burgers_fd} with the final time $T = 2.0$ using the Midpoint(1,2,2), IMEX-SSP3(3,3,2) and IMEX-SSP3(4,3,3)schemes.}
\label{tab:8}
\end{table}

\subsubsection{Effect of choice of $f$ and $g$} \label{sec:linear_pde_fd_f_g_choice}

The choice of $f$ and $g$ is often not obvious for complex problems. In this example, we  show how the components of the error estimate capture the effects of the choice of the explicit and implicit parts. To illustrate this, we reverse the choice of $f$ and $g$ in the linear PDE of \S \ref{sec:linear_pde_fd}. That is, we set $f$ as the term arising from $\gamma u_{xx}$ and $g$ as term arising from $\sin(2 \pi x) u_x$. The results for this choice for the same QoI as in \S \ref{sec:linear_pde_fd}  are shown in Table~\ref{tab:9} for $\gamma = 0.075$, $k_n = 1/40$ and $h = 1/20$. We observe that this choice of $f$ and $g$ leads to instability for the Midpoint(1,2,2)  scheme, which has a large error relative to the IMEX-SSP schemes. Moreover, we note that the contribution of the component $E2$, which corresponds to the error due to integration of the explicit term, dominates the error estimate and is significantly larger than the other two contributions, $E1$ and $E3$, as expected. These results are in contrast to the results in Tables~\ref{tab:1}--\ref{tab:4}, where the component $E2$ is not the dominant term.

  \begin{table}[!ht]
\centering
\begin{tabular}{ l| c |c |c |c | c }
\toprule
Scheme & Comp. Err. &Eff. Ratio & $E1$ & $E2$ & $E3$\\
\midrule
Mid(1,2,2)            & 1.75E+06&  1.00& 1.16E+05 & 1.64E+06 & 1.27E+04       \\
SSP3(3,3,2)         &  -4.48E-02&     1.00 & -6.22E-03 & -3.92E-02  & 6.45E-04     \\
SSP3(4,3,3)         &  -7.14E-02 &   1.00&  -1.95E-02  &-5.16E-02  &-2.11E-04         \\
\bottomrule
\end{tabular}
\caption{Results for  the problem in \S \ref{sec:linear_pde_fd_f_g_choice} with the final time $T = 1.0$ and an unstable choice of $f$ and $g$ using the Midpoint(1,2,2), IMEX-SSP3(3,3,2) and IMEX-SSP3(4,3,3)schemes.}
\label{tab:9}
\end{table}

\subsubsection{1D Magnetohydrodynamic Problem}
\label{sec:allfven_wave}
For our final result, we consider a one-dimensional simplification of the three-dimensional resistive magnetohydrodynamic equations \cite{GoedbloedPoedts2004,ShadidEtAlMHD2015}.  
For clarity of notation that follows we consider a right-handed coordinate system with axis definitions $(\chi,\zeta,\varsigma)$ that corresponds to the typical $(x,y,z)$ Cartesian system.	
This problem is an analytic asymptotic model for the propagation of an Alf\'en wave in a viscous conducting fluid that fills the half space above the $\chi$-axis.  The solution is of the form $\mbfB = (B(\zeta,t),B_0,0)$ and $\mbfv = (v(\zeta,t),0,0)$.  The resistive MHD equations reduce  to the transient term, the Lorentz force term, and the viscous stress term in the $\chi$-momentum equation,
\begin{equation}
\label{eq:alfven_v}
 \frac{\partial v}{\partial t} = \frac{B_0}{\rho}  \frac{\partial B}{\partial \zeta} + \frac{\mu}{\rho} \frac{\partial^2 v}{\partial \zeta^2},
\end{equation}
and the transient term, the induction transport term, and the magnetic diffusion term for the $\chi$-magnetic induction equation,
\begin{equation}
\label{eq:alfven_B}
\frac{\partial B}{\partial t} = B_0  \frac{\partial v}{\partial \zeta} +  \frac{\eta}{\mu_0} \frac{\partial^2 B}{\partial \zeta^2}.
\end{equation}
These equations are often used as a test for more complicated three-dimensional MHD codes, as the 1D counterpart has an analytic solution,
\begin{align}
	v(\zeta,t) &= \frac{U}{4}\left[e^{\frac{-A_0 \zeta}{d}} \left(1-\erf \left(\frac{\zeta-A_0 t}{2\sqrt{d t}}\right)\right)-\erf \left(\frac{\zeta-A_0 t}{2\sqrt{d t}} \right)\right] + \\
&\qquad \frac{1}{4} U\left[e^{\frac{A_0 \zeta}{d}} \left(1-\erf\left(\frac{\zeta+A_0 t}{2\sqrt{dt}}\right)\right)-\erf \left(\frac{\zeta + A_0 t}{2\sqrt{d t}}\right) +2\right],	\notag\\
	B(\zeta,t) &= -\frac{1}{4} e^{\frac{-A_0 \zeta}{d}} \left(-1 + e^{\frac{A_0 \zeta}{d}}\right) U \sqrt{\mu \rho} \left[\erfc \left(\frac{\zeta-A_0 t}{2\sqrt{d t}}\right) + e^{\frac{A_0 \zeta}{d}} \erfc \left(\frac{\zeta + A_0 t}{2\sqrt{d t}}\right)\right],
\end{align}
where $d = \eta/\mu_0$. 

We consider this problem with parameters $B_0 = 10$ and all other parameters are set to 1. The initial conditions are chosen to be $v = B = 0$. Plots of the true solution and IMEX solution for the velocity variable, $v$, at different times for two IMEX different schemes are shown in Figure~\ref{fig:alfven_solns}. 
For the IMEX solutions, the boundary conditions are obtained from the exact solution.   We discretize the spatial domain with discretization parameter $h = 5 \by 10^{-3}$ to obtain a system of the form \eqref{eq:model_ode}, which we solve to final time $T = 0.1$ with time step $k_n = 1 \by 10^{-3}$. The second-order operators were integrated implicitly and the first-order operators (corresponding to the Alf\'en wave) were treated explicitly, see \S\ref{sec:alfven_choice_f_g_a}.  The figure indicates that the solution for the Midpoint(1,2,2) scheme in Figure~\ref{fig:alfven_mid} is unstable, whereas the solution for the IMEX-SSP3(3,3,2) in Figure~\ref{fig:alfven_332} scheme is quite accurate. The plot for the IMEX-SSP3(4,3,3) is similar to Figure~\ref{fig:alfven_332}. The accuracy of the solutions indicated by the figures is also quantitatively identified by the error estimates in \S\ref{sec:alfven_choice_f_g_a}.

\begin{figure}[!ht]
\centering
\begin{subfigure}[!ht]{0.46\textwidth}
\centering
\includegraphics[width = \textwidth]{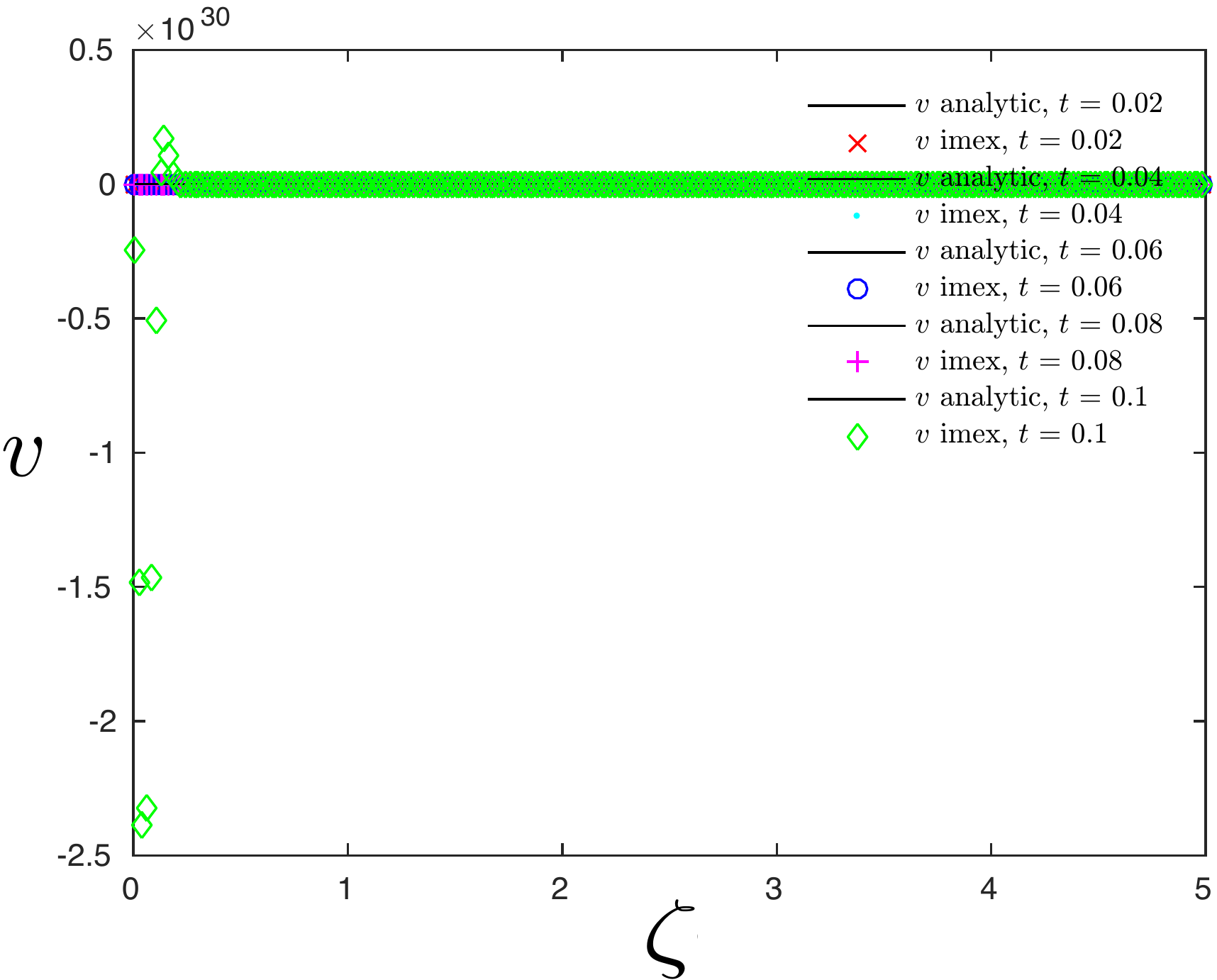} 
\caption{}
\label{fig:alfven_mid}
\end{subfigure}
\hfill
\begin{subfigure}[!ht]{0.46\textwidth}
\includegraphics[width = \textwidth]{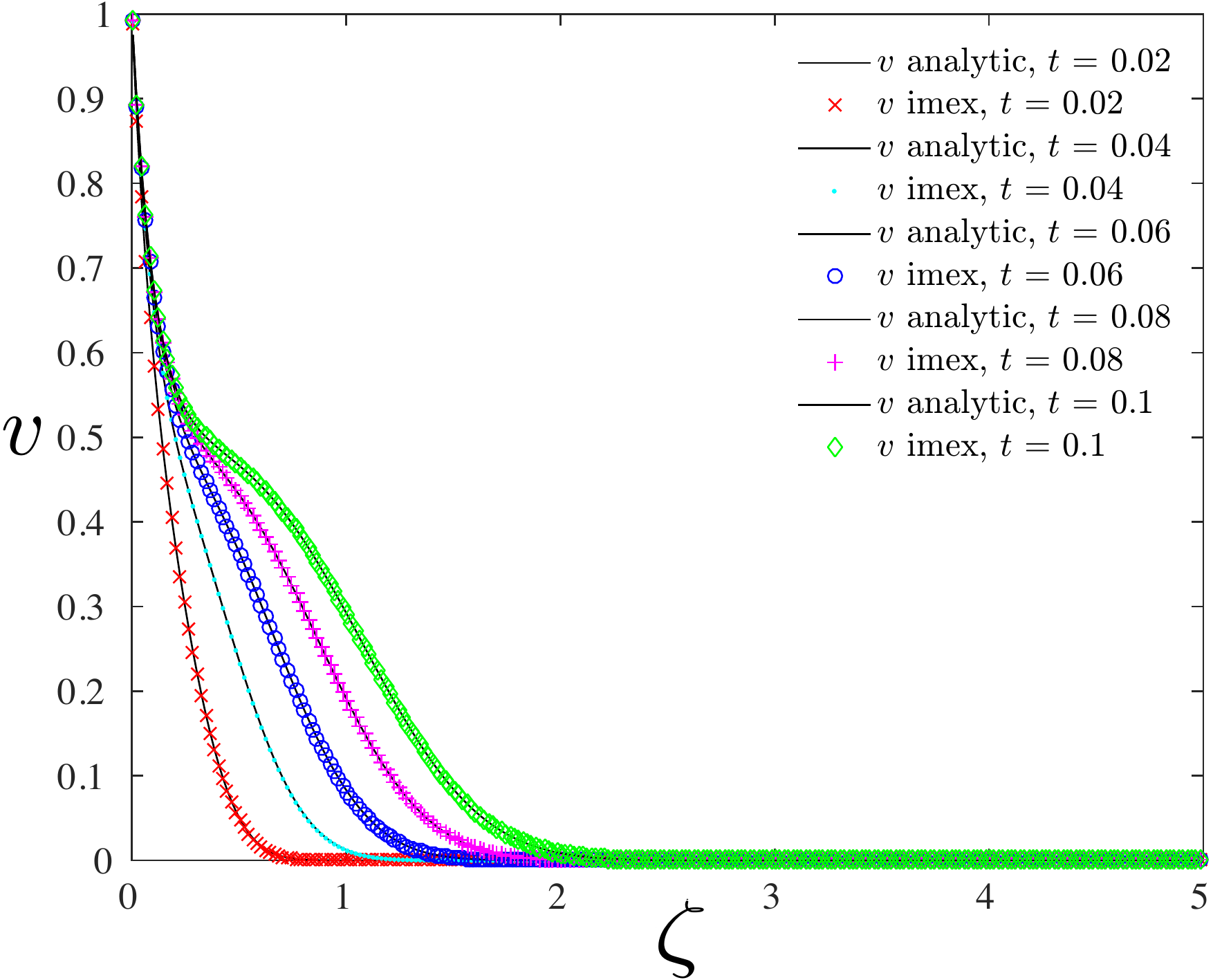} 
\caption{}
\label{fig:alfven_332}
\end{subfigure}
\caption{Plot of $v$ for the problem in \eqref{eq:alfven_v} and \eqref{eq:alfven_B} at different times. The analytic solution, labeled $v$ analytic, is a solid black line. The IMEX solution is label $v$ imex. (a) Solution obtained using the Midpoint(1,2,2) scheme. (b) Solution obtained using the IMEX-SSP3(3,3,2) scheme. The solution for the midpoint scheme exhibits instabilities whereas the solution for the IMEX-SSP3(3,3,2) scheme is quite accurate. The plot for the IMEX-SSP3(4,3,3) is similar to the one for the  IMEX-SSP3(3,3,2) scheme.}
\label{fig:alfven_solns}
\end{figure}

We choose a quantity-of-interest $\int_0^L v \,d\zeta$ where $\zeta\in [0,L]$ belongs to the spatial domain considered for the problem. We further decompose the solution as  $y = [y_v, y_B]^\top$ with the components  $y_v$ and $y_B$ corresponding to the (spatially discretized) variables $v$ and $B$ respectively.  Similarly, we decompose $f = [f_v, f_B]^\top$, $g = [g_v, g_B]^\top$ and the adjoint solution $\phi = [\phi_v, \phi_B]$.  For all runs, in the context of the induction equation $f_B$ corresponds to the spatial discretization of the term $B_0  \frac{\partial v}{\partial \zeta}$  and $g_B$ corresponds to the spatial discretization of the term $ \frac{\eta}{\mu_0} \frac{\partial^2 B}{\partial \zeta^2}$, with $f_v$ and $g_v$ defined below for each example.
	
	Finally, we decompose the error components from Theorem \ref{thm:err_rep} as,
	\begin{align}
		E1 = E1_v + E1_B, \quad E2 = E2_v + E2_B, \quad E3 =  E3_v + E3_B,	
	\end{align}
	where,
	\begin{align}
			E1_v &= \sum_{n=0}^{N-1}  \langle - \dot{Y_v}, \phi_v  - \pi_n \phi_v \rangle_{I_n} + \langle f_v(\mathcal{I} Y) , \phi_v - \pi_n \phi_v \rangle_{{I_n},Q^f} + \\
			&\qquad \langle g_v(\mathcal{I} Y)  , \phi_v- \pi_n \phi_v \rangle_{{I_n},Q^g},  \notag\\
			E2_v &= \sum_{n=0}^{N-1} \langle f_v(Y) , \phi_v \rangle_{I_n} - \langle f_v(\mathcal{I} Y) , \phi_v \rangle_{{I_n},Q^f}, \\
			E3_v &= \sum_{n=0}^{N-1} \langle g_v(Y), \phi_v \rangle_{I_n} -   \langle g_v(\mathcal{I} Y)  , \phi_v \rangle_{{I_n},Q^g},
	\end{align}
where $Y = [Y_v, Y_B]^\top$ is the IMEX solution. The definitions for the error contributions corresponding to the $B$ variables follow in a similar manner.

\paragraph{Implicit and Explicit components for the $v$ equation}
\label{sec:alfven_choice_f_g_a}
	For our first result we split the $v$ components of the right hand side into both implicit and explicit parts.  We choose $f_v$ to correspond with $\frac{B_0}{\rho}  \frac{\partial B}{\partial \zeta} $ and $g_v$ to correspond with $\frac{\mu}{\rho} \frac{\partial^2 v}{\partial \zeta^2}$.  The results are shown in Tables~\ref{tab:10a} and \ref{tab:10b}.  We observe that the error estimate has effectivity ratio close to one, even when the actual error is quite large, as is the case for the Midpoint(1,2,2) scheme. The unstable solution for this scheme is depicted in Figure \ref{fig:alfven_mid}. 

\begin{table}[!ht]
\centering
\begin{tabular}{ l| c |c |c |c | c }
\toprule
Scheme & Comp. Err. &Eff. Ratio & $E1$ & $E2$ & $E3$\\
\midrule
Mid(1,2,2)            &1.73e+27 & 1.00 & 8.01 E+26 & -2.10E+26& 1.13E+27\\
SSP3(3,3,2)         &  -3.30E-02 & 1.00& -3.033E-05& 1.59E-01& -1.92E-01\\
SSP3(4,3,3)         &  5.19E-04 & 1.00 & 1.785E-04 & 8.71E-05& 2.54E-04\\
\bottomrule
\end{tabular}
\caption{Results for  the problem in \S \ref{sec:allfven_wave} with the choice of $f$ and $g$ given in \ref{sec:alfven_choice_f_g_a} and final time $T = 0.1$ using the Midpoint(1,2,2), IMEX-SSP3(3,3,2) and IMEX-SSP3(4,3,3)schemes.}
\label{tab:10a}
\end{table}

\begin{table}[!ht]
\centering
\begin{tabular}{ l|c| c |c |c |c | c }
\toprule
Scheme &$E1_v$ & $E1_B$ & $E2_v$ & $E2_B$ & $E3_v$ & $E3_B$\\
\midrule
Mid(1,2,2)            &8.07E+26 & -6.04E+24& -1.93E+26& -1.70E+25& 9.51E+26& 1.87E+26\\ 
SSP3(3,3,2)         &  -2.92E-05& -1.06E-06& 2.85E-03& 1.56E-01& -1.89E-01& -3.02E03\\
SSP3(4,3,3)         &  1.73E-04 &4.91E-06 & 4.90E-05& 3.80E-05& 3.76E-04& -1.22E-04\\
\bottomrule
\end{tabular}
\caption{Different components of the error contributions corresponding to the results in Table~\ref{tab:10a}.}
\label{tab:10b}
\end{table}

\paragraph{Implicit components only for the $v$ equation}
\label{sec:alfven_choice_f_g_b}
	For the next result, in the context of the momentum equation we choose all the $v$ components of the right hand side to be handled implicitly, thus setting $f_v = 0$ and letting $g_v$ correspond with $\frac{B_0}{\rho}  \frac{\partial B}{\partial \zeta}  + \frac{\mu}{\rho} \frac{\partial^2 v}{\partial \zeta^2}$.  In this case both the Lorentz force term and the viscous stress  are integrated implicitly. 
Tables~\ref{tab:11a} and \ref{tab:11b} show the results, and now we observe that all three schemes are associated with much
more accurate results. Apparently in this case representing even one component of the first order terms that compose the Alfven wave helps to stabilize the Midpoint(1,2,2) method. Clearly, these results demonstrate the accuracy of the  error estimate.

\begin{table}[!ht]
\centering
\begin{tabular}{ l| c |c |c |c | c }
\toprule
Scheme & Comp. Err. &Eff. Ratio & $E1$ & $E2$ & $E3$\\
\midrule
Mid(1,2,2)            &-4.88E-04 & 1.00 & 3.88E-06 & 3.22E-04 & -8.14E-04\\
SSP3(3,3,2)         & -3.27E-02& 1.00 & -3.57E-07& 1.55E-01& -1.87E-01\\ 
SSP3(4,3,3)         &  6.7628e-04& 1.00 & 1.1075e-06 & 2.0620e-04 & 4.6897e-04\\
\bottomrule
\end{tabular}
\caption{Results for  the problem in \S \ref{sec:allfven_wave} with the choice of $f$ and $g$ given in \ref{sec:alfven_choice_f_g_b} and final time $T = 0.1$ using the Midpoint(1,2,2), IMEX-SSP3(3,3,2) and IMEX-SSP3(4,3,3)schemes. }
\label{tab:11a}
\end{table}

\begin{table}[!ht]
\centering
\begin{tabular}{ l|c| c |c |c |c | c }
\toprule
Scheme &$E1_v$ & $E1_B$ & $E2_v$ & $E2_B$ & $E3_v$ & $E3_B$\\
\midrule
Mid(1,2,2)            &1.38E-05 &-9.97E-06 & 0 & 3.22E-04 &  -3.10E-03  & 2.28E-03\\
SSP3(3,3,2)         &  1.03E-06 & -1.39E-06 & 0 &  1.55E-01 & -1.85E-01 & -2.87E-03\\
SSP3(4,3,3)         & -3.89E-06 &  5.00E-06  & 0 &   2.06E-04  & 4.51E-04 &  1.73E-05\\
\bottomrule
\end{tabular}
\caption{Different components of the error contributions corresponding to the results in Table~\ref{tab:11a}. }
\label{tab:11b}
\end{table}

\section{Conclusions}
	We present adjoint-based \emph{a posteriori} error estimation for multi-stage Runge-Kutta IMEX schemes.  These estimates are achieved by representing the IMEX scheme as a finite element method, which uses particular quadratures to obtain nodal equivalence with the IMEX scheme.  This provides us with an approximation that equals the IMEX approximation at the nodes, but is defined for the entire temporal domain.  We then use this approximation to estimate the error in the IMEX approximation.  In addition, our analysis distinguishes between error due to the discretization of the temporal domain, the explicit portion of the scheme and the implicit portion of the scheme.  This splitting of the error into different contributions allows us to determine what portion of the method is most responsible for inaccuracy, and can inform the user as to the best course method to reduce the error.

\section*{Acknowledgments}  
J. Chaudhry's work is supported in part by the Department of Energy (DE-SC0009324) and by Sandia National Laboratories: Laboratory  Directed Research and Development (LDRD) Funding under Academic Alliance Program FY2016. 
J. N. Shadid's work was partially supported by the DOE Office of Science Applied Mathematics Program at Sandia National Laboratories under contract DE-AC04-94AL85000. The authors will also like to thank Prof. Don Estep from Colorado State University for discussing the ideas presented here.

\section*{References}
\bibliographystyle{elsarticle-num}
\bibliography{imex_rk}

\end{document}